\documentclass[12pt]{article}
\usepackage{amsthm}
\usepackage{amssymb}
\usepackage{amsfonts, mathrsfs}
\usepackage{amsmath}
\usepackage{algorithm}
\usepackage{color}\usepackage{graphicx}

\newcommand{\pr}{\mathbf{Pr}}

\newcommand{\p}[1]{{\mathbf P}\left\{#1\right\}}

\newcommand{\set}[1]{\left\{ #1 \right\}}

 \newcommand{\bag}{\begin{align}}
\newcommand{\bags}{\begin{align*}}
\newcommand{\eag}{\end{align*}}
\newcommand{\eags}{\end{align*}}

\newtheorem{thm}{Theorem}
\newtheorem{lem}[thm]{Lemma}

\newcommand\cC{\mathcal C}

\newcommand\cG{\mathcal G}

\newcommand{\beq}[2]{\begin{equation}\label{#1}#2\end{equation}}
\newcommand{\brac}[1]{\left(#1\right)}
\newcommand{\bfrac}[2]{\brac{\frac{#1}{#2}}}
\begin{document}
\title{How many randomly colored edges make a randomly colored dense graph rainbow hamiltonian or rainbow connected?}
\author{ Michael Anastos\footnote{Department of Mathematical Sciences, Carnegie Mellon University, Pittsburgh PA 15213, USA, Email: manastos@andrew.cmu.edu,  Research supported in part by NSF grant DMS0753472} and Alan Frieze\footnote{Department of Mathematical Sciences, Carnegie Mellon University, Pittsburgh PA 15213, USA, Email: alan@random.math.cmu.edu, Research
supported in part by NSF grant DMS0753472}}

\maketitle
\begin{abstract}
In this paper we study the randomly edge colored graph that is obtained by adding randomly colored  random edges to an arbitrary randomly edge colored dense graph. In particular we ask how many colors and how many random edges are needed so that the resultant graph contains a fixed number of edge disjoint rainbow Hamilton cycles. We also ask when in the resultant graph every pair of vertices is connected by a rainbow path. 
\end{abstract}

\section{Introduction}
In this paper we study the following random graph model:
We start with a graph $H=(V,E)$ and a set $R$  of $m$ edges chosen uniformly at random from 
$\binom{[n]}{2}\backslash E$, that is from the edges not found in $H$. We then add to $H$ the edges  in $R$ to get the graph
$$G_{H,m}= (V,E\cup R).$$
After this, we color every edge of $G_{H,m}$ independently and uniformly at random with a color from $[r]$. We denote the resultant colored graph by 
$$G_{H,m}^r=(V,E\cup R,c) .$$
Here $c:E\cup R \mapsto [r]$ is the function that assigns to every edge in $E\cup R$ its color.
\vspace{3mm}
\\The  random graph model $G_{H,m}$ was first introduced  by Bohman, Frieze and Martin 
in \cite{BFM}. It can be consider as an extension of the Erd\H{o}s-R\'enyi model which we can retrieve from $G_{H,m}$ by setting $H=\emptyset$.
The main motivation for this model is the following: Let $\cal{G}$ be some class of graphs and $\cal{P}$ be a property, an increasing property in our case, that is satisfied by almost all the members of $\cal{G}$. The following question arises. For any graph $G\in \cal{G}$, suppose we perturb slightly its edge set at random, by adding a few random edges.
How many are needed so that the result is a graph that satisfies property $\cal{P}$. In \cite{BFM} they study the case where $\cal{G}$ is the set of graphs of minimum degree $\delta n$, $\delta>0$
and $\cal{P}$ is the property of a graph having a Hamilton cycle. They show that a linear number of random edges suffices in order to make any member of $\cal{G}$ Hamiltonian
w.h.p.\@ \footnote{We say that a sequence of events ${\cal{E}}_n$ holds \emph{with high probability} if $\p{{\cal{E}}_n} \rightarrow 1$ as $n \rightarrow \infty$.}
They also point out that for $\delta< 0.5$, complete bipartite graphs with bipartitions of sizes $\delta n$ and $(1-\delta) n$ need a linear number of random edges in order to became Hamiltonian. Here it is worth mentioning then in the Erd\H{o}s-R\'enyi random graph, $G(n,m)$, the threshold for Hamiltonicity is $(\log n +\log \log n)n$. 
$G_{H,m}$ has since been studied in a number of other contexts, see for example \cite{BFKM}, \cite{BMPP}, \cite{KKS1}, \cite{KKS2} and \cite{MM}. 
\vspace{3mm}
\\In this paper we enhance this model by randomly coloring edges. We $[r]$-color the edges of $G_{H,m}$ independently and uniformly at random and we denote the resultand graph by $G_{H,m}^r$.
We then ask about the existence of rainbow Hamilton cycle in  $G_{H,m}$ and whether $G_{H,m}^r$ is rainbow connected.
\vspace{3mm}
\\A Hamilton cycle is called rainbow if no color appears twice on its edges. It was shown by Frieze and Loh \cite{FL}  and by Ferber and Krivelevich \cite{FKr} that for $m\geq (1+o(1))(\log n+\log \log n)n$ if we color $G(n,m)$ randomly with $(1+o(1))n$ colors then $G(n,m)$ contains a rainbow Hamilton cycle. This implies that if we randomly $[(1+o(1))n]$-color a typical graph of minimum degree $\delta n$ then w.h.p.\@ the resultant graph will contain a rainbow Hamilton cycle.  
In comparison, as mentioned earlier, a graph of linear minimum degree needs only a linear number of edges in order to became Hamiltonian.
\vspace{3mm}
\\We say a graph is rainbow connected if every pair of its vertices are connected by a rainbow path. For a fixed graph $G$ of minimum degree $\delta$ is known that $\frac{\log \delta }{\delta}n(1+f(d))$ colors are needed in order to color it such that the resultant graph is rainbow connected.
\vspace{3mm}
\\The class of graphs of interest in our case will be the graphs on $n$ vertices, of minimum degree $\delta n$ with $\delta >0$ which we denote by ${\cal{G}}(n,\delta)$. For the rest of this paper we let $0<\delta <0.5 $ and $H$ be an arbitrary member of ${\cal{G}}(n,\delta)$. We also let 
\beq{defs}{
\theta=\theta(\delta)=-\log \delta \hspace{5mm}\text{ and }\hspace{5mm}
t=t(\delta)=\min\bigg\{\frac{\delta n}{260}, \frac{n}{1000+200\theta}\bigg\}.
}
Our first theorem builds on the Hamiltonicity result of \cite{BFM}.
\begin{thm}\label{HamRainbow}
Let $m\geq \min \big\{ (435+75\theta)tn,  \big|\binom{[n]}{2} \setminus E(H)\big| \big\}$ and 
$r\geq(120+20\theta)n$ then,  w.h.p.\@ $G^r_{H,m}$ contains $t$ edge disjoint rainbow Hamilton cycles.
\end{thm}
The above theorem states that a linear number of edges and a linear number of colors suffice in order for $G_{H,m}^r$ to have rainbow-Hamilton cycle. In addition it says that if you require multiple number of edge disjoint rainbow Hamilton cycles it suffices to multiply the number of random edges that are added.
\vspace{3mm}
\\Let $G$ be a graph of minimum degree $k$. A $k$-out random subgraph of $G$, denoted $G_{k-out}$, can be generated by adding independently at random $k$ edges incident to vertex $v$ for every $v\in V$. 
\vspace{3mm}
\\ We partition $H$ into two subgraphs as follows. We include every edge of $H$ into $E(H')$ independently with probability $p=\frac{1}{20}$.  We then set $E(H'')=E(H)\setminus E(H')$. Since $H$ has minimum degree $\delta n$
and $ p \cdot \delta n \gg \log n$, the Chernoff bounds imply that w.h.p.\@ $E(H')$ has minimum degree $\delta(H')\geq\frac{\delta n}{21}$ and maximum degree $\Delta(H')\leq \frac{n}{19}$. We partition $H$ so that it will be easier to expose the relative randomness in stages.
\vspace{3mm}
\\ To prove Theorem \ref{HamRainbow} we first reprove a non-colored version of it. Namely we show
\begin{thm}\label{Ham}
Let  $Q\subset H''\cup R$ be such that $|Q|=(81+15\theta)n$ and Q is distributed as a random subset of $\binom{[n]}{2}$ of the corresponding size. Then w.h.p.\@ $H_{6-out}'\cup Q$ is Hamiltonian
\end{thm}
To get Theorem \ref{HamRainbow} from Theorem \ref{Ham} we use a result of Ferber et al given in \cite{FKMS} and stated as Theorem \ref{auxiliary} below. We use it in order to extract $t$ rainbow subgraphs from $H'$ each of which has the properties of $H_{6-out}$ needed in the proof of Theorem \ref{Ham}. And each of these subgraphs can become hamiltonian by adding to it $(81+15\theta)n$ random edges. We then argue that by suitably refining those edges the Hamilton cycles will be rainbow.  
\vspace{3mm}
\\ The following Theorem concerns the rainbow connectivity of  $G^r_{H,m}$.
\begin{thm}\label{RainbowConnectivity}
For rainbow connectivity the following holds:
\begin{description}
\item[(i)] If $r=3$ and $m\geq 60\delta^{-2}\log n$ then  w.h.p.\@ $G^r_{H,m}$ is rainbow connected.
\item[(ii)] For $\delta \leq 0.1$ there exist $H\in \mathcal{G}(n,\delta)$ such that if $m \leq  0.5 \log n$  then w.h.p. $G^4_{H,m}$ is not rainbow connected.
\item[(iii)] If $r=7$ and $m=\omega(1)$ then w.h.p.\@   $G^r_{H,m}$ is rainbow connected.
\end{description}
\end{thm}
The rest of the paper is divided as follows. In Section 2 we give the proof of Theorem \ref{HamRainbow} and in Section 3 we give the proof of Theorem \ref{RainbowConnectivity}. We close with Section 4.
\section{Proof of Theorem \ref{HamRainbow}}
\subsection{Proof of Theorem \ref{Ham}}
We first split $Q$ into two sets $Q_1,Q_2$ of sizes $(45+15 \theta)n$ and $36n$ respectively. We then show that $H_{6-out}'\cup Q_1$ is connected and has good expansion properties. Then we apply a standard P\'osa rotation argument to show that $E(H')\cup Q_1\cup Q_2$ spans a Hamilton cycle. 
\begin{lem}\label{koutproperties}
With probability $1-o(n^{-2})$, the following hold: 
\begin{enumerate}
\item $H_{6-out}'\cup Q_1$ is connected,
\item for every $S\subseteq V$ such that $|S|\leq n/5$ we have 
that $|N_{H_{6-out}'\cup Q_1}(S)|>2|S|$. 
\end{enumerate}
\end{lem}
\begin{proof}
We start by examing the second property for ``small" sets.
\vspace{3mm}
\\ \emph{Claim 1:} With probability $1-o(n^{-2})$  every $S\subset V$ such that $|S|\leq \delta^2n/200$ satisfies
 \\$|N_{H_{6-out}'}(S)| >2|S|.$
\vspace{3mm}
\\ \emph{Proof of Claim 1:} Let $S,T\subset V$ be such that $|T|=2|S|$ and 
$|S|\leq \delta^2 n/200$. In $H'$ every vertex in $S$ has at most $3|S|\leq 3\delta^2 n/200$ out neighbors in $S\cup T$. Furthermore observe that given $H'$ every set of 6 edges adjacent to $v$ in $H$ is equaly likely to be chosen by $v$ during the constraction of $H_{6-out}'$.
Thus 
$$\pr(N_{H_{6-out}'}(v)\subseteq S\cup T) \leq \binom{3|S|}{6}\bigg \backslash \binom{\delta n}{6}\leq \bigg(\frac{3|S|}{\delta n} \bigg)^6.$$      
Therefore
\begin{align*}
\pr(\text{Claim 1 is violated}) &\leq \sum_{s=1}^{\delta^2 n/200}\binom{n}{s}\binom{n}{2s} \bigg(\frac{3s}{\delta n} \bigg)^{6s}\\
&\leq \sum_{s=1}^{\delta^2 n/200}\bigg(\frac{en}{s}\bigg)^s\bigg(\frac{en}{2s}\bigg)^{2s} \bigg(\frac{3s}{\delta n} \bigg)^{6s}\\
&\leq \sum_{s=1}^{\delta^2 n/200}\bigg(\frac{e^3\cdot3^6s^3}{\delta^6 n^3}\bigg)^s=o(n^{-2}).
\end{align*}
{\parindent 0in\emph{End of proof of Claim 1.}}

\vspace{3mm}
Now we examine the second property for ``large" sets.  Here we are going to use the edges from $Q_1$. Let $\Delta(H')$ be the maximum degree of $H'$. Then $\Delta(H')\leq n/19$ w.h.p.

\vspace{3mm}
\emph{Claim 2:}  With probability $1-o(n^{-2})$  every $S\subseteq V$ such that
 $ \delta^2n/200< |S| \leq n/5$ satisfies $|N_{H_{6-out}'\cup Q_1}(S)| >2|S|$.
\vspace{3mm}
\\ \emph{Proof of Claim 2:} For 
$\delta^2n/200\leq |S| \leq n/5$ by considering only the edges in $Q_1$ we have 
\begin{align*}
\pr(\text{Claim 1 is violated})
&\leq \sum_{s=\frac{\delta^2n}{200}}^{\frac{n}{5}} \binom{n}{s}\binom{n}{2s} \bigg(1-\frac{s(n-3s)- \Delta(H')\cdot s }{\binom{n}{2}} \bigg)^{|Q_1|}\\
&\leq  \sum_{s=\frac{\delta^2n}{200}}^{\frac{n}{5}} \bigg( \frac{en}{s}\bigg)^{s} \bigg( \frac{en}{2s}\bigg)^{2s} e^{-\frac{s(n-3s-\Delta(H'))|Q_1|}{\binom{n}{2}}}\\
&\leq \sum_{s=\frac{\delta^2n}{200}}^{\frac{n}{5}} 
\bigg( \frac{e^3n^3}{4s^3}e^{-\frac{2|Q_1|}{5n}}\bigg)^{s}  
\leq \sum_{s=\frac{\delta^2n}{200}}^{\frac{n}{5}} 
\bigg( \frac{e^3\cdot 200^3}{4\delta ^6}e^{-(18+6\theta)}\bigg)^{s}=o(n^{-2}).  
\end{align*}
Claims 1 and 2 imply the second property of our Lemma. At the same time Claim 1 implies  that every connected component of $H_{6-out}'$ has size at least $\frac{\delta^2n}{200}$.
In the event that any two of the at most $200\delta^{-2}$ components of $H_{6-out}'$ are connected by an edge in $Q_1$ we have that $H_{6-out} \cup Q_1$ is connected. 
Observe for any two disjoint sets $S_1,S_2$ of size at least $\delta^2 n/200$, the Chernoff bounds imply that 
$$\pr (S_1 \cup S_2 \text{ span } \geq S_1||S_2|/10 \text{ edges in } H') =\exp \set{-O(\delta^2 n^2)}.$$
Since there are at most $2^n$ choices for each of $S_1,S_2$ we have that w.h.p\@ every pair of components $S_1,S_2$ of $H_{6-out}'$ spans at most $|S_1||S_2|/10$ edges. Therefore
\begin{align*}
\pr(H_{6-out}'\cup Q_1 \text{ it not connected })
\leq \binom{200\delta^{-2}}{2} \bigg(1-\frac{9}{10}\cdot\bfrac{\delta^2n}{200}^2\cdot\frac{1}{\binom{n}{2}}\bigg)^{|Q_1|}
=o(n^{-2}).
\end{align*}
\end{proof}
Our next Lemma builds on Lemma \ref{koutproperties} and completes the proof of Theorem \ref{Ham}.  It is basically an adaptation of P\'osa's argument to our setting. 
\begin{lem}\label{clasicHam}
$$\pr(H_{6-out}'\cup Q_1 \cup Q_2\text{ is not Hamiltonian})=o(n^{-2}).$$
\end{lem}
\begin{proof}
Let $Q_2=\{e_1,...,e_{|Q_2|}\}$ and set $G_i=(H_{6-out}\cup Q_1)\cup\set{e_1,...,e_i}$.
Assume that $G_i$ is not Hamiltonian and  consider a longest path $P_i$ in $G_i,i\geq 0$. Let $x,y$ be the end-vertices of $P$. Given $yv$ where $v$ is an interior vertex of $P_i$ we can obtain a new longest path $P_i' = x..vy..w$ where $w$ is the neighbor of $v$ on $P_i$ between $v$ and $y$. In such a case we say that $P_i'$ is obtained from $P_i$ by a rotation with the endpoint $x$ being the fixed end-vertex.

Let $END_i(x;P_i)$ be the set of end-vertices of longest paths of $G_i$ that can be obtained from $P_i$ by a sequence of rotations that keep $x$ as the fixed end-vertex. Thereafter for $z\in END_i(x;P_i)$ let $P_i(x,z)$ be a path that has end-vertices $x,z$ and can be obtain form $P_i$ by a sequence of rotations that keep $x$ as the fixed end-vertex. Observe that since $G_i$ is connected but not Hamiltonian for $z\in END_i(x;P_i)$ and $z'\in END_i(z;P_i(x,z))$ neither $xz$ nor $zz'$ belong to $G_i$ since otherwise we can close the path into a cycle that is not Hamiltonian and then use the connectivity of $G_i$ to get a longer path than $P$. At the same time it follows from P\'osa \cite{Posa} 
$$|N(END_i(x,P_i))| < 2|END_i(x,P_i)|.$$
Moreover for every $z\in END_i(x;P_i)$
$$|N(END_i(z,P_i(x,z)))|< 2|END_i(z,P_i(x,z))|.$$
As a consequence, since Lemma \ref{koutproperties} states that $\forall S\subset V$ $|S|\leq n/5$ we have $2|S|\geq N_{H_{6-out} \cup Q_1}(S)$, we get that $|END_i(x,P_i)| \geq \frac{n}{5}$.
\vspace{3mm}
\\Let $E_i=\set{\set{z,z'}\notin H': z\in END_i(x;P_i) \text{ and } z'\in END_i(z;P_i(x,z))}$. $H'$ has maximum degree $n/19$. Furthermore $|END(x,P_i)| \geq n/5$ and for every $z\in END_i(x,P_i)$ we have $|END_i(z;P_i(x,z))| \geq n/5$. Hence 
$$|E_i|\geq \frac{1}{2}\cdot \frac{n}{5}\bigg( \frac{n}{5}-\frac{n}{19}\bigg) \geq \frac{n^2}{70}.$$
Now let $Y_{i+1}$ be the indicator that $e_{i+1}\in E_i$ and $Z=\sum_{i=1}^{|Q_2|} Y_i$. Then $\pr (Y_i=1) \geq 1/35$. Therefore $Z$ dominates the binomial $Bin(|Q_2|,1/35)$. Thus since $|Q_2|/35=36n/35$, the  Chernoff bound implies that $\pr(Z\leq n)=e^{-\Omega(n)}=o(n^{-2})$. Hence $G_{|Q_2|}=H_{6-out}'\cup Q_1 \cup Q_2$ is Hamiltonian with the required probability.
\end{proof}
\subsection{Partitioning $H'$}
We are now ready to apply Theorem \ref{auxiliary}, given below and partition $H'$ into $t+1$ edge disjoint subgraphs. $t$ of them will be rainbow and will satisfy the property discribed in Theorem \ref{Ham}.  The final step will be to show that for each $i\in [t]$ we can set aside a random subset $R_i\subseteq R$, of size $(81-15\delta)n$ such that $H_i \cup R_i$ is rainbow.
\begin{thm}[\cite{FKMS}]\label{auxiliary} 
Let $\epsilon>0$ be a constant, $k\geq 2$ be an integer and $\mathcal{P}$ be a monotone increasing graph property. Let $F$ be a graph on $n$ vertices with minimum
degree $\delta(F)=\omega(\log n)$ whose edges are colored independently and uniformly at random from $[kn]$. Then, w.h.p.\ $F$ can be partitioned into $F=F_0\cup F_1\cup\dots\cup F_r$ such that the following holds.
  \begin{enumerate}
    \item $F_0,F_1,\ldots,F_r$ are edge-disjoint subgraphs of $F$,
    \item $r=\left(1-\varepsilon\right)\frac{\delta(F)}{2k}$,
   	\item For every $1\leq i\leq r$, $E(F_i)$ is rainbow and of size at most $kn$, and
    \item For every $1\leq i\leq r$, $\Pr\left[F_{k\text{-out}}
    \text{ satisfies } \mathcal{P}\right]\leq \Pr\left[F_i\text{ satisfies } \mathcal{P}\right]+n^{-\omega(1)}$.
  \end{enumerate}
\end{thm}
\qed

By monotone, we mean here that if $F_{k-out}$ satisfies $\mathcal{P}$ then adding edges to $F_{k-out}$ gives a graph that also satisfies $\mathcal{P}$.

We apply the above Theorem with $F=H'$ and $r=t,k=6$ and $\mathcal{P}$ the property that the addition of $(81+15\delta)n$ random edges from $\binom{[n]}{2}\setminus H'$ makes the graph Hamiltonian with probability at least $1-n^{-2}$. It follows from Theorem \ref{auxiliary} and Lemma \ref{clasicHam} that we can finish our proof by showing that w.h.p. we can pair each of the $H_i,i\in[t]$ with a random subset $Q_i\subseteq E(H'')\cup R$ of size $(81+15\delta)n$ such that each $H_i\cup Q_i$ is rainbow. Furthermore, we will color $H'$ using more than $6n$ colors and this will not invalidate the use of Theorem \ref{auxiliary}.
\subsection{Partitioning $H'' \cup R$}
We can assume that $m=\min \set{ (435+75\theta)tn, \bar{m}_H},\bar{m}_H= \big|\binom{[n]}{2} \setminus E(H)\big| \big\}$. We start by extracting $t$ disjoint sets from $E(H'') \cup R $. We choose $m|E''|/\bar{m}_H$ edges uniformly at random from $E(H'')$ and add them to $R$. Let $\{e_1,...,e_{m'}\}$ be a random permutation of these edges. Since w.h.p.\@ $H'$ has maximum degree $\delta n/19$ we have that $m' \geq  (435+75\theta)tn$. Furthermore $\{e_1,...,e_{m'}\}$ is distributed as a random subset of $\binom{[n]}{2}\setminus E'$ os size $m'$. We let $Q_i'= \{e_{(435+75\theta)in+1},...,e_{(435+75\theta)(i+1)n}\}.$ It follows from Lemma \ref{clasicHam} that any subset $A_i$ of $Q_i'$ that satisfies: i) $A_i \cup H_i'$ is rainbow and ii) $|A_i'|=(81+15\theta)n$ satisfies the requirements for $Q_i$. In the case that at least $(87+15\theta)n$ colors appear in $Q_i'$ such a set exists. (The extra $6n$ needed for Lemma \ref{clasicHam} deals with the colors of $H_i$). Finally the probability that fewer colors appear is bounded by
\begin{multline*}
\binom{r}{(87+15\theta)n} \bigg(\frac{(87+15\theta)n}{r}\bigg)^{(435+75\theta)n}
\\ \leq  \bigg(\frac{er}{(87+15\theta)n}\bigg)^{(87+15\theta)n}  \bigg(\frac{(87+15\theta)n}{r}\bigg)^{(435+75\theta) n}
\\ =e^{(87+15\theta)n} \bigg(\frac{(87+15\theta)n}{(120+20\theta)n}\bigg)^{(348+60\theta)n} \leq \bigg[ e \bfrac{3}{4}^4\bigg]^{(87+15\theta)n} =o\bigg(\frac{1}{n}\bigg).
\end{multline*}   
This completes the proof of Theorems \ref{HamRainbow} and \ref{Ham}.\\
\qed 
\section{Proof of Theorem \ref{RainbowConnectivity}}
\subsection{Proof of Theorem \ref{RainbowConnectivity} (i)}
We will show that if $R$ is large enough then w.h.p., for any pair of vertices $u,v \in V$ there are many edges in $R$ between their neighborhoods $N(u),N(v)$. It will follow that w.h.p. there is a rainbow path of length 3 from $u$ to $v$.
\vspace{3mm}
\\Let $R=\set{r_1,...,r_m}$. Let $\cC$ be the event that $G$ is rainbow connected. For $u,v\in V$ let $\cC_3 (u,v)$ be the event that there exists a rainbow path $u,u_0,v_0,v$ with $\set{u,u_0},\set{v,v_0}\in E(H)$ and $\set{u_0,v_0}\in R$. Furthermore let $B(u,v)$ be the event that there exist fewer than $10 \log n$ such paths in $G$. Given $r_1,...,r_{i-1}$ either there exist $10\log n$ such paths or $r_i$ creates such a path with probability at least $ \delta n(\delta n-10\log n)/ \binom{n}{2} \geq \delta^2.$
Therefore the Chernoff bound implies that
\begin{align*}
\pr ( B(u,v))&\leq \pr(Binomial(60\delta^{-2}\log n,\delta^2)<10\log n)
\leq \exp\set{- \frac12 \cdot \frac19\cdot 60\log n}\leq n^{-3}.
\end{align*}
Observe that a path of length 3 in $G$ is rainbow with probability $1\cdot \frac{2}{3}\cdot \frac{1}{3}=\frac{2}{9}$ since we may assign any color to its first edge, then any of the other two colors to its second edge and finally the remaining color to its third edge. Thus
\begin{multline*}
\pr(\urcorner \cC)\leq \sum_{u,v\in V:u\neq v} \pr(\urcorner	\cC(u,v) )
\leq \sum_{u,v\in V:u\neq v}
\pr(	B(u,v) ) +\pr(\urcorner	\cC(u,v)|\urcorner B(u,v) )\\
\leq \sum_{u,v\in V:u\neq v} \brac{n^{-3}  + \bigg(1-\frac{2}{9}\bigg)^{10 \log n}} 
\leq n^2 \big(n^{-3} +n^{-20/9} \big)=o(1).
\end{multline*}
\subsection{Proof of Theorem \ref{RainbowConnectivity} (ii)}
Our counterexample will consist of 2 disjoint copies of $G(0.5n,p)$ with $p=0.11$. We will show that if $|R|$ is not sufficiently large then it will not cover every vertex in the neigborhoods of some vertices in either copies. 
\vspace{3mm}
\\Let $\delta\leq 0.1$. Partition $V$ into 2 sets $V_1,V_2$ each of size $0.5n$.
Then generate $H$ by including in $E(H)$ every edge in $V_1\times V_1$ or in $V_2\times V_2$ independently with probability $0.22$. Since $0.22\cdot 0.5n= 0.11n$, the  Chernoff bounds imply that for all $v\in V$ the degree of $v$, $d(v)$ satisfies $0.1n<d(v)<0.12n$. In particular w.h.p. $H\in \cG(n,0.1)$.

In the case that $\exists v\in V_1$ and $u\in V_2$ such that no edge in $R$ has an endpoint in each of  $(\{v\}\cup N(v))\times(\{u\}\cup N(u))$ then $u$ and $v$ are at distance at least 5 in $G$. Since any such path cannot be rainbow when is colored by four colors we have that $G$ is not rainbow connected.

Observe that w.h.p. $R$ covers sets $R_1\subset V_1$ and $R_2\subset V_2$ each of size at most $\log n$. Therefore a vertex in $V_1\setminus R_1$ has at least one  neighbor in $R_1$ independently with probability\\
$1-(1-0.22)^{|R_1|} \leq 1-n^{-1/2}.$ Therefore 
$$\pr(\exists v\in V_1: (\{v\}\cup N(v))\cap R=\emptyset)\geq 1-  (1-n^{-1/2})^{0.5n}=1-o(1).$$ 
Similarly, $\pr(\exists v\in V_2: (\{v\}\cup N(v))\cap R=\emptyset)=1-o(1).$ Hence w.h.p.\@ $G$ is not rainbow connected when $r=4$.
\subsection{Proof of Theorem \ref{RainbowConnectivity} (iii)}
We extract from $V$ a small set of vertices $S$ such that for every $v\in V$ there exists $s\in S$ that shares many neighbors with $v$ in $H$ (see Lemma \ref{connections}). We then show that any two vertices in $S$ are connected by a rainbow path of length 3. We extend these paths into many paths of length 7 to show that w.h.p. $G^7_{H,m}$ is rainbow connected.
\begin{lem}\label{connections}
Let $G \in \mathcal{G}(n,\delta)$. Then there exists $S\subset V$ satisfying the following conditions: 
\begin{enumerate}
\item $|S|\leq 2/\delta$.
\item $\forall v \in V\setminus S$, there exists $s\in S$ such that $|N(v) \cap N(s)| \geq \delta^2 n/4$.
\end{enumerate}
\end{lem}
\begin{proof}
Let $S$ be a maximal subset of $V$ such that for every $v,w \in S$ we have $|N(v) \cap N(w)| < \delta^2 n/4$. Then the maximality of $S$ implies that $S$ satisfies the second condition of our Lemma.  Then either $|S|< 2/\delta$ or there exist $S_1\subset S$ of size $\lceil 2/\delta \rceil$. In the latter case we have 
\begin{multline*}
n=|V| \geq \ \left|\bigcup_{s\in S_1} N(s)\right| \geq \sum_{s\in S_1}|N(s)|- \sum_{s_1\neq s_2\in S_1} |N(s_1)\cap N(s_2)| \\
\geq \delta n |S_1| - \frac{\delta^2 n}{4}\cdot \frac{|S_1|^2}{2} = \delta |S_1| n \bigg(1-\frac{\delta |S_1|}{8}\bigg)>n.
\end{multline*} 
Contradiction.
\end{proof}
{\bf{Proof of  Theorem \ref{RainbowConnectivity} (iii)}}. 
Let $S$ be a set satisfying the conditions of Lemma \ref{connections}.
For $v\in V$ let $s_v\in S$ be such that $|N(v) \cap N(s)| \geq \delta^2 n/4$.
Let ${\cal{J}}_S$ be the event that every pair of vertices $s_1,s_2$ are joined by three edge disjoint rainbow paths. Since $|S|=O(1)$ and each vertex in $S$ has $\Omega(n)$ neighbors and $m=\omega(n)$ and $r=7$ we have $P({\cal{J}}_S)=1-o(1)$. Given ${\cal{J}}_S$  occuring let $v_1,v_2 \in S$.
Then for any pair of vertices $v_1,v_2$, there is a rainbow path $P_{v_1,v_2}$ of length 3  from $s_{v_1}$ to $s_{v_2}$ not containing $v_1,v_2$. Assume that $v_1,v_2 \notin S$ and that they share fewer than $\log^2n$ neighbors. Let $J_{v_1,v_2}$ be the event that $P_{v_1,v_2}$ can be extended to a rainbow path 
from $v_1$ to $v_2$. Assume that $P_{v_1,v_2}$ uses colors $5,6,7$. Then there will be a rainbow path from $v_1$ to $v_2$ if there is a vertex $w\in N(v_1)\cap N(s_{v_1})$ such that edge $\set{v_1,w}$ gets color 1 and edge $\set{w,s_{v_1}}$ gets color 2 and colors 3,4 are similarly used for $v_2,s_{v_2}$. It follows that
$$\Pr(J_{v_1,v_2} \text{ does not occur } ) \leq  2\brac{1-\bfrac{1}{7}^2}^{\delta^2 n/4 -\log^2 n}=o(n^{-3}).$$
The remaining cases for $v_1,v_2$ follow in a similar manner. Taking a union bound over $v_1,v_2$ give us Theorem \ref{RainbowConnectivity} (iii). 
\section{Conclusion}
We have extended the notion of adding random edges to dense graphs and asking probabilistic questions to that of adding randomly colored edges. The most interesting question for us that is left open by the above analysis is the gap between 4 and 7 in Theorem \ref{RainbowConnectivity} (iii).

\end{document}